\documentclass[10pt, reqno, a4paper]{amsart}

\usepackage[latin1]{inputenc}
\usepackage{amsfonts}
\usepackage{amsmath}
\usepackage{amssymb}
\usepackage{amsthm}
\usepackage[T1]{fontenc}
\usepackage{enumerate}
\usepackage{verbatim}
\usepackage{lmodern}
\usepackage{mathtools}

\RequirePackage[dvipsnames,usenames]{color}  
\definecolor{VeryDarkGreen}{rgb}{0,0.18,0.08}
\definecolor{VeryDarkBrown}{rgb}{0.12,0.08,0.04}
\usepackage[pdftex]{hyperref}
\hypersetup{
bookmarks,
bookmarksdepth=2,
bookmarksopen,
bookmarksnumbered,
pdfstartview=FitH,
colorlinks,backref,hyperindex,
linkcolor=VeryDarkBrown,
citecolor=VeryDarkGreen,
urlcolor=VeryDarkGreen
}
\usepackage[all]{xy}
\SelectTips{cm}{10}

\DeclareMathOperator{\Hom}{Hom}

\DeclareMathOperator{\id}{id}
\DeclareMathOperator{\coker}{coker}
\DeclareMathOperator{\colim}{colim}

\DeclareMathOperator{\Sol}{Sol}

\DeclareMathOperator{\Spec}{Spec}

\title[Intermediate extensions and smooth pullbacks commute]{Intermediate extensions of perverse constructible $\mathbb{F}_p$-sheaves commute with smooth pullbacks}

\begin{document}

\swapnumbers
\theoremstyle{plain}
\newtheorem{Le}{Lemma}[section]
\newtheorem{Ko}[Le]{Corollary}
\newtheorem{Theo}[Le]{Theorem}
\newtheorem*{TheoB}{Theorem}
\newtheorem{Prop}[Le]{Proposition}
\newtheorem*{PropB}{Proposition}
\newtheorem{Con}[Le]{Conjecture}
\theoremstyle{definition}
\newtheorem{Def}[Le]{Definition}
\newtheorem*{DefB}{Definition}
\newtheorem{Bem}[Le]{Remark}
\newtheorem{Bsp}[Le]{Example}
\newtheorem*{BspB}{Example}
\newtheorem{Be}[Le]{Observation}
\newtheorem{Sit}[Le]{Situation}
\newtheorem{Que}[Le]{Question}
\newtheorem{Dis}[Le]{Discussion}
\newtheorem{Prob}[Le]{Problem}
\newtheorem*{Konv}{Conventions}
\newtheorem{claim}[Le]{Claim}
\newtheorem{Bla}[Le]{}

\def\cocoa{{\hbox{\rm C\kern-.13em o\kern-.07em C\kern-.13em o\kern-.15em
A}}}

\author{Axel St\"abler}

\address{Axel St\"abler\\Department of Mathematics, University of Michigan, Ann Arbor, MI 48109, USA}
\curraddr{Johannes Gutenberg-Universit\"at Mainz\\ Fachbereich 08\\
Staudingerweg 9\\
55099 Mainz\\
Germany}

\email{staebler@uni-mainz.de}

\date{\today}


\begin{abstract}
We prove that intermediate extensions of perverse constructible $\mathbb{F}_p$-sheaves commute with smooth pullbacks for embeddable schemes over a field of characteristic $p$. Along the way we also prove that the equivalence of categories of Cartier crystals with unit $R[F]$-modules commutes with $f^!$ for a smooth morphism $f: X \to Y$ of embeddable schemes.
\end{abstract}

\maketitle

\section{Introduction}
First, let us consider the category of perverse constructible $\mathbb{F}_\ell$-sheaves on $X_{\'et}$, where $X$ is, say, of finite type over a field $k$ of characteristic $p >0 $ and $\ell \neq p$ is a prime. Then it is well-known that every object of this category has finite length. It is then desirable to understand the simple objects of this category. This is where the intermediate extension $j_{!\ast}$ comes into play, where $j$ is a locally closed immersion. Any simple perverse sheaf on $X$ is of the form $j_{!\ast} \mathcal{F}$ for some locally closed immersion $j: U \to X$ and $\mathcal{F}$ a simple locally constant perverse sheaf on $U$.
The intermediate extension $j_{!\ast} \mathcal{F}$ of $\mathcal{F}$ along $j$ in turn is given as the image of the natural map $j_! \mathcal{F} \to j_\ast \mathcal{F}$. Equivalently, one may define it as the smallest subobject $\mathcal{S}$ of $j_\ast \mathcal{F}$ for which $j^{-1} \mathcal{S} = \mathcal{F}$. It is a consequence of the smooth base change theorem that intermediate extensions commute with smooth pullbacks. That is, if $f: X \to Y$ is a smooth morphism and $j: U \to Y$ a locally closed immersion, then $f^{-1} j_{!\ast} \cong j'_{! \ast} f'^{-1}$ where $f'$ and $j'$ are the base changes of $f$ and $j$.
\[\begin{xy} \xymatrix{ f^{-1}(U) \ar[r]^{j'} \ar[d]^{f'} &X \ar[d]^f \\ U \ar[r]^j & Y }  \end{xy} \]

The purpose of this note is to investigate the behavior of intermediate extensions under smooth pullbacks in the case where $\ell = p$. In this case Gabber (\cite{gabbertstructures}) introduced a notion of perversity on the category of constructible sheaves. Emerton and Kisin introduce in \cite{emertonkisinrhunitfcrys} the category of locally finitely generated unit $F$-modules and construct an anti-equivalence, called $\Sol$, from this category to the category of perverse constructible $\mathbb{F}_p$-sheaves. In \cite[Corollary 4.2.2]{emertonkisinintrorhunitfcrys} intermediate extensions are introduced via locally finitely generated unit $F$-modules. Namely, for a locally closed immersion $j: U \to X$  and a locally finitely generated unit $F$-module $M$ on $U$ one defines $j_{!+} M$ as the smallest subobject $S$ of $j_+ M$ that satisfies $j^! S = j^! j_+ M$, where $j_+$ corresponds to $j_!$ on the constructible side and similarly $j^!$ corresponds to $j^{-1}$. The authors also show that any simple perverse sheaf is obtained in this way.

We prove that the intermediate extension commutes with smooth pullbacks for morphisms between schemes that are embeddable into smooth schemes. Our proof uses the anti-equivalent category of so-called \emph{Cartier crystals}. This anti-equivalence is obtained by composing the anti-equivalence $\Sol$ with an equivalence $\Sigma$ from Cartier crystals to finitely generated unit $F$-modules which was constructed by Blickle and B\"ockle in \cite{blickleboecklecartierfiniteness} in the smooth case and shown to extend to the embeddable case in \cite{schedlmeiercartierpervers} by Schedlmeier.

The category of Cartier crystals also admits an intrinsic description of intermediate extensions as shown in (\cite{schedlmeiercartierpervers}). Moreover, in Cartier crystals intermediate extensions relate to so-called \emph{test modules} (see \cite[Definition 3.1, Remark 3.3 ]{blicklep-etestideale}) under certain circumstances. It is known by work of the author that the formation of test modules does commute with smooth twisted inverse images (\cite[Corollary 4.8]{staeblerunitftestmoduln}). We will show that intermediate extensions in Cartier crystals commute with smooth twisted inverse images (Theorem \ref{SmoothPullbackIE}) which is technically much simpler than the corresponding result for test modules.

In order to achieve the corresponding result in the category of perverse sheaves we also have to show that under the anti-equivalence of categories between Cartier crystals and perverse constructible $\mathbb{F}_p$-sheaves the twisted inverse image $f^!$ of a smooth morphism $f: X \to Y$ corresponds to the pullback $f^{-1}$ for schemes $X, Y$ that are embeddable into a smooth scheme. While this is certainly expected there is currently no proof of this available which is why we include one here.

Finally, let us remark that the assumption that our schemes are embeddable is only necessary for the equivalence to work. It is not required in order to show that intermediate extensions in Cartier crystals commute with smooth twisted inverse images. Conjecturally, the equivalence with constructible sheaves should also be true in a more general context.

We will review the necessary theory of Cartier modules and crystals as well as the anti-equivalence between Cartier crystals and perverse constructible sheaves in the next section (i.e.\ Section \ref{Cartiercrystals}). Then we will proceed to show that this anti-equivalence interchanges $f^!$ and $f^{-1}$ in Section \ref{EquivalenceShriekPullbackCompatible}. This is somewhat technical and we encourage the reader to skip ahead to Section \ref{IESection} on a first reading, where we prove that intermediate extensions commute with twisted inverse images in Cartier crystals. 

\subsection*{Conventions} Throughout $F$ denotes the absolute Frobenius morphism and schemes are assumed to be noetherian. Given a scheme $X$ and an $\mathcal{O}_X$-module $M$ we denote by $F^e_\ast M$ the $\mathcal{O}_X$-module which as a sheaf of abelian groups is $M$ but with module structure induced by the $e$th iterate of the Frobenius, i.e. $r \cdot m = r^{p^e}m$ for sections $r \in \mathcal{O}_X(U)$ and $m \in M(U)$ for $U \subseteq X$ open. We will work exclusively in positive prime characteristic and assume that schemes $X$ considered are \emph{$F$-finite}, i.e.\ the Frobenius morphism $F: X \to X$ is a finite morphism. Given an $F$-finite field $k$, we call a $k$-scheme $X$ \emph{embeddable} if there exists a closed immersion $X \to X'$ with $X'$ smooth over $k$.

We also will have to assume that our schemes admit a notion of relative dimension. Therefore we further restrict our attention to schemes for which the irreducible components coincide with the connected components. If $f: X \to Y$ is a morphism and $X_1, \ldots, X_r$ are the irreducible components of $X$, then $f(X_i)$ is irreducible and thus contained in a unique irreducible component of $Y$ which we denote by $Y_i$. We define the relative dimension of $f$ as the tuple $(\dim X_1 - \dim Y_1, \ldots, \dim X_r - \dim Y_r)$. In fact, once we know that intermediate extensions commute with base change with respect to open immersions this case also reduces to the irreducible case but we will not need this (see Remark \ref{IrredComponents}).

If we fix some $F$-finite field $k$ then the category of embeddable $k$-schemes (for which irreducible components $=$ connected components) form a full subcategory of the category of $k$-schemes which we denote by $Sch_{\text{emb}}$. In particular, morphisms in $Sch_{\text{emb}}$ are assumed to be $k$-linear.

\subsection*{Acknowledgements} The author was supported by grant STA 1478/1-1 of the Deutsche Forschungsgemeinschaft (DFG). I thank M.\ Blickle for useful discussions, T.\ Schedlmeier as well as the referee for a careful reading of this article and useful comments.

\section{Cartier crystals and unit $F$-modules}
\label{Cartiercrystals}

In this section we explain the necessary background around Cartier crystals and unit $F$-modules. Unless otherwise noted $X$ is an $F$-finite noetherian scheme. Note that any scheme (essentially) of finite type over an $F$-finite field (e.g.\ a perfect field) is $F$-finite.

\subsection{Cartier modules and crystals}

We refer the reader to \cite{blickleboecklecartierfiniteness} for a gentle introduction. Background about the derived category of Cartier crystals may be found in \cite{blickleboecklecartiercrystals}. Both sources restrict to the case of a single structural map which corresponds to the case of a Cartier algebra generated by a single element in degree $1$. For a discussion of Cartier modules in the setting of a Cartier algebra as in Definition \ref{DefCartierAlgebra} below see \cite{blicklestaeblerfunctorialtestmodules}.

\begin{Def}
\label{DefCartierAlgebra}
A \emph{Cartier algebra} $\mathcal{C}$ on a scheme $X$ is a graded sheaf of rings $\bigoplus_{e \geq 0} \mathcal{C}_e$ with an $\mathcal{O}_X$-bimodule structure which satisfies $r \kappa = \kappa r^{p^e}$ for any local section $r$ and any local homogeneous element $\kappa$ of degree $e$. Moreover, we assume that $\mathcal{C}_0 = \mathcal{O}_X$.
\end{Def}

A \emph{Cartier module} (or $\mathcal{C}$-module if we want to stress the algebra) is a left $\mathcal{C}$-module. We say that a Cartier module $M$ is (quasi-)coherent if the underlying $\mathcal{O}_X$-module is so.

Of central importance is the notion of \emph{nilpotence}.

\begin{Def}
We say that a coherent $\mathcal{C}$-module $M$ is \emph{nilpotent} if there exists $e \geq 0$ such that $(\mathcal{C}_+)^e M = 0$. A morphism $\varphi: M \to N$ of coherent $\mathcal{C}$-modules is a \emph{nil-isomorphism} if $\ker \varphi$ and $\coker \varphi$ are nilpotent.
\end{Def}

The category of nilpotent Cartier submodules is a Serre subcategory of coherent Cartier modules (see \cite[Lemma 2.11]{blickleboecklecartierfiniteness} for the case that $\mathcal{C} = \mathcal{O}_X\langle\kappa\rangle$, the general case is similar). Hence, we may localize at this subcategory (i.e.\ one formally inverts all nil-isomorphisms) and call the localized category \emph{Cartier crystals}. More precisely, objects in Cartier crystals are the same as in Cartier modules. A morphism  $\alpha: M \to N$ is a diagram of the form $M \xleftarrow{\beta} P \to N$, where $\beta$ is a nil-isomorphism. We will denote the category of (coherent) Cartier crystals by $Crys(X)$.

Given a Cartier module $M$ an element $\kappa$ of $\mathcal{C}_e(U)$, where $U \subseteq X$ is any open, acts on $M(U)$ as an $\mathcal{O}_X(U)$-linear map $\Hom_{\mathcal{O}_X(U)}(F_\ast^e M(U), M(U))$. The most important example (and the reader may restrict to this case on a first reading) is the case where $\mathcal{C}$ is principal, i.e. $\mathcal{C} = \mathcal{O}_X\langle\kappa\rangle$ generated by a single element $\kappa$ in degree $1$ (say). Then the datum of a coherent Cartier module corresponds to a coherent $\mathcal{O}_X$-module $M$ together with an $\mathcal{O}_X$-linear map $F_\ast M \to M$. By abuse of notation we will denote this map again by $\kappa$. Equivalently, one may view $M$ as a \emph{right module} over the non-commutative polynomial ring quotient $\mathcal{O}_X[F]$ which is given by \[\mathcal{O}_X[F](U) := \mathcal{O}_X(U)\{F\}/\mathcal{O}_X(U)\langle r^p F - F r\, \vert\, r \in \mathcal{O}_X(U)\rangle.\] for any $U \subseteq X$ open.

In the case where we have a single map as our Cartier structure we will also write $(M, \kappa)$ to denote the datum of a Cartier module. In this case $M$ is nilpotent if and only if $\kappa^e M = 0$ for all $e \gg 0$.

In order to be able to construct injective resolutions and to construct a functor $f_\ast$ of Cartier crystals for a finite type morphism $f$ we will also need to weaken the coherence assumption to quasi-coherence. The notion of nilpotence is then too strong and we need a slight variant of this:

\begin{Def}
\label{Deflocnilpotent}
Let $M$ be a quasi-coherent $\mathcal{C}$-module. We call $M$ \emph{locally nilpotent} if there are nilpotent Cartier submodules $M_e$ ($e \in \mathbb{N}$) such that $\bigcup_{e \geq 0} M_e = M$. A morphism $\varphi:M \to N$ is called a \emph{nil-isomorphism} if both $\ker \varphi$ and $\coker \varphi$ are locally nilpotent.
\end{Def}

Note that for a coherent Cartier module the notions of local nilpotence and nilpotence agree. The category of locally nilpotent Cartier modules again form a Serre subcategory of quasi-coherent Cartier modules. We may thus again pass to the localized category of \emph{quasi-coherent Cartier crystals}.

\begin{Theo}
\label{CartierShriek}
Let $f: X \to Y$ be a morphism of $F$-finite schemes and $\mathcal{C}_Y$ a Cartier algebra on $Y$. Then $\mathcal{C}_X = f^\ast \mathcal{C}_Y$ is a Cartier algebra on $X$. Moreover, in each of the following cases the twisted inverse image functor $f^!$ on quasi-coherent modules induces a functor from quasi-coherent $\mathcal{C}_Y$-modules to quasi-coherent $\mathcal{C}_X$-modules:
\begin{enumerate}[(a)]
\item{$f$ is smooth, so that $f^! \bullet = f^\ast \bullet \otimes \omega_{X/Y}$. In particular, if $f$ is \'etale, then $f^! = f^\ast$.}
\item{$f$ is finite, so that $Rf^! = \bar{f}^\ast R\Hom(f_\ast \mathcal{O}_Y, \bullet)$.}
\end{enumerate}
In each of these cases $f^!$ preserves coherence and (local) nilpotence. In particular, one obtains induced functors on the category of (quasi-)coherent Cartier crystals.
\end{Theo}
\begin{proof}
See \cite[Section 5]{blicklestaeblerfunctorialtestmodules}.
\end{proof}

It will be important later on to know how the Cartier structure is constructed in Theorem \ref{CartierShriek} so let us elaborate on this. It is sufficient to descibe this action locally, so we may pick an open affine $U \subseteq Y$ and an open affine $V \subseteq f^{-1}(U)$.

We may then view the action of a homogeneous element $\kappa \in \mathcal{C}_U$ of degree $e$ as a morphism $\kappa \in \Hom_{\mathcal{O}_Y(U)}(F_\ast^e M(U), M(U))$. In order to understand the action of $\mathcal{C}_X$ it suffices to describe to corresponding action of $\kappa \otimes 1$ as a morphism $\Hom_{\mathcal{O}_X(V)}(F_\ast^e f^!M(V), f^!M(V))$.
\begin{enumerate}[(a)]
\item{If $f$ is \'etale, then the morphism $\varphi: f^\ast F^e_\ast M \to F^e_\ast f^\ast M$ given on local sections by $m \otimes s \mapsto m \otimes s^{p^e}$ is an isomorphism. Now one defines the Cartier structure via \[\begin{xy} \xymatrix@1{F^e_\ast f^\ast M \ar[r]^{\varphi^{-1}}& f^\ast F_\ast^e M \ar[r]^{\kappa \otimes id} & f^\ast M.} \end{xy}\]
In the smooth case, we may locally factor $f$ as \[ \begin{xy} \xymatrix@1{\Spec S \ar[r]^{\varphi} &  \mathbb{A}^n_R \ar[r]^{g}& \Spec R }\end{xy}\] with $\varphi$ \'etale. Then the Cartier structure on $f^! M$ is the one induced from $\varphi^! g^! M$, where the Cartier structure on $g^! M= g^\ast M \otimes \omega_g$ is given by \[m \otimes x_1^{t_1} \cdots x_n^{t_n} dx \mapsto \kappa(m) \otimes x_1^{\frac{t_1 + 1}{p^e} -1} \cdots x_n^{\frac{t_n + 1}{p^e} -1} dx.\] Here $x_1, \ldots, x_n$ are local coordinates, $dx = dx_1 \wedge \cdots \wedge dx_n$ and if $r$ is not an integer, then $x_i^{r} = 0$  by convention.}
\item{Recall that $\bar{f}:(X, \mathcal{O}_X) \to (Y, f_\ast \mathcal{O}_X)$ is the canonical flat map of ringed spaces. By our assumption the Frobenius $F$ is a finite morphism. Hence, by duality for a finite morphism $\kappa: F_\ast^e M \to M$ corresponds to a map $M \to {F^e}^! M$. We then have an isomorphism $f^! {F^e}^! M \cong {F^e}^! f^! M$ which, by applying duality again, induces a morphism $F_\ast^e f^!M \to f^!M$. This is the action of $\kappa \otimes 1$.

The important case for us is, where $f$ is a complete intersection, say of relative dimension $n$. In this case only $R^n\Hom(f_\ast \mathcal{O}_Y, M)$ is non-zero. Specifically, if $f: \Spec S \to \Spec R$ and the corresponding regular sequence is given by $x_1, \ldots, x_n \in R$, then $R^n f^! M = M/(x_1, \ldots, x_n)M$ where the Cartier action of $\kappa \otimes 1$ is given by $m \mapsto \kappa x_1^{p^e-1} \cdots x_n^{p^e-1} m$. (see \cite[Example 3.3.12]{blickleboecklecartiercrystals} for details).}
\end{enumerate}

\begin{Theo}
Let $f: X \to Y$ be a morphism of finite type between $F$-finite schemes (e.g.\ an open immersion). Let $\mathcal{C}_Y$ be a Cartier algebra finitely generated\footnote{By this we mean that there is a finite covering by open affines $\Spec R_i$ of $Y$ such that $\mathcal{C}_Y\vert_{\Spec R_i}$ is a quotient of $R_i\{F_1, \ldots, F_n\}/R_i\langle F_j r^{p^{e_j}}  - r F_j  \, \vert \, r \in R \rangle$ for natural numbers $e_j$.} as an $\mathcal{O}_Y$-algebra and assume that there is an isomorphism $\varphi: \mathcal{C}_X \to f^\ast \mathcal{C}Y$. Then the functor $f_\ast$ on quasi-coherent modules induces a functor $f_\ast$ from quasi-coherent $\mathcal{C}_X$-modules to quasi-coherent $\mathcal{C}_Y$-modules. Moreover, this functor preserves (local) nilpotence and restricts to a functor from coherent $\mathcal{C}_X$-crystals to \emph{coherent} $\mathcal{C}_Y$-crystals.
\end{Theo}
\begin{proof}
See  \cite[Section 2.3 and Theorem 3.2.14]{blickleboecklecartiercrystals} for the case that $\mathcal{C}_Y$ is principally generated and \cite[Theorem 7.10]{blicklestaeblerfunctorialtestmodules} for the general case.
\end{proof}

Let us again explain how the Cartier structure looks like. The natural map $\mathcal{C}_Y \to f_\ast f^\ast \mathcal{C}_Y$ is a morphism of Cartier algebras (see \cite[Lemma 5.11]{blicklestaeblerfunctorialtestmodules}. Thus, having fixed some isomorphism $\mathcal{C}_X \cong f^\ast \mathcal{C}_Y$ as above, we obtain an action of $\mathcal{C}_Y$ on $f_\ast M$.

We now come to the notion of intermediate extension in the category of Cartier crystals.

\begin{Def}
\label{IntermediateExtensionDef}
Let $X$ be an $F$-finite scheme and $j: U \to X$ a locally closed immersion. Given a Cartier crystal $M$ on $U$ we define the \emph{intermediate extension} $j_{!\ast} M$ (of $M$ along $j$) to be the smallest subcrystal $N$ of $j_\ast M$ for which $j^! N = j^! j_\ast M = M$.
\end{Def}

Note that we may factor $j$ as $f \circ g$ with $f$ a closed immersion and $g$ an open immersion. If $X$ is embeddable into a smooth scheme and the Cartier structure is principal, then we have $j^! j_\ast \cong (fg)^! (fg)_\ast \cong g^! f^! f_\ast g_\ast \cong g^! g_\ast$, where we use \cite[Theorem 5.12 (c)]{schedlmeiercartierpervers} and \cite[Theorem 4.1.2]{blickleboecklecartiercrystals}.

The intermediate extension always exists for open immersions (see \cite[Theorem 6.13]{schedlmeiercartierpervers} for a proof in the case that $\mathcal{C} = \mathcal{O}_X\langle\kappa\rangle$, the general case being similar). We prove existence in the case $\mathcal{C} = \mathcal{O}_X\langle\kappa\rangle$ for a locally closed immersion in Lemma \ref{LclosedIEOpenIE} below.

Next we discuss two notions of derived categories for Cartier crystals.

\begin{Def}
Let $X$ be an $F$-finite scheme and $\mathcal{C}_X = \mathcal{O}_X\langle \kappa \rangle$.
\begin{enumerate}[(a)]
\item{We denote the \emph{bounded derived category of quasi-coherent Cartier crystals with cohomology in coherent Cartier crystals} by $D^b_{crys}(QCrys(X))$ or simply by $D^b_{crys}(X)$.}
\item{We denote the \emph{bounded derived category of coherent crystals} by $D^b(Crys(X))$.}
\end{enumerate}
\end{Def}

\begin{Prop}
\label{CrystalsDerivedCatEquivalent}
If $X$ is $F$-finite and regular or embeddable, then the natural functor $ D^b(Crys(X)) \to D^b_{crys}(X)$ is an equivalence of categories.
\end{Prop}
\begin{proof}
See \cite{blickleboecklecartierduality} for the regular case and \cite{schedlmeiercartierpervers} for the embeddable case.
\end{proof}

Assume that $\mathcal{C}$ is of the form $\mathcal{O}_X\langle \kappa \rangle$ with $\kappa$ of degree $1$. Then we may equivalently view a coherent Cartier module as a coherent $\mathcal{O}_X$-module endowed with a map $F_\ast M \to M$. If $X$ is $F$-finite and embeddable, then the category of Cartier crystals on $X$ is anti-equivalent to the category of perverse constructible $\mathbb{F}_p$-sheaves via the intermediate category of locally finitely generated unit $F$-modules. Thus we review this category next.

\subsection{unit $F$-modules}
For an introduction to unit $F$-modules see \cite{emertonkisinintrorhunitfcrys}, a more elaborate treatment may be found in \cite{emertonkisinrhunitfcrys}. For the embeddable case see \cite{schedlmeiercartierpervers}.
Our naming convention will differ from \cite{emertonkisinrhunitfcrys} in that we will only consider unit $F$-modules that are locally finitely generated in this article and refer to them simply as \emph{unit $F$-modules}.

\begin{Def}
\begin{enumerate}[(a)]
\item{Let $X$ be a smooth scheme over a field $k$.  Let $M$ be a quasi-coherent $\mathcal{O}_X$-module and  $\Phi: F^\ast M \to M$ an isomorphism. Consider a \emph{coherent} $\mathcal{O}_X$-module $N$ together with an injective map $\varphi: N \to F^\ast N$. We can then consider $A =  \colim_{e \geq 0} F^{e\ast} N$, where the direct system is given by the composition of appropriate $F^{e\ast} \varphi$. The map $\varphi$ induces an isomorphism $\Phi_A: F^\ast A \to A$. We call $(N, \varphi)$ a \emph{root} of $M$ if there exists an isomorphism $\alpha: A \to M$ that is compatible with $\Phi_A$ and $\Phi$, i.e.\ $\alpha \circ \Phi_A = \Phi \circ F^\ast \alpha$.}
\item{We can now define the \emph{category of unit $F$-modules} on $X$ (we will denote this by \emph{unit $F(X)$}). Its objects are pairs $(M, \Phi)$, where $M$ is a quasi-coherent $\mathcal{O}_X$-module and $\Phi: F^\ast M \to M$ is an isomorphism. Further, we require that $(M, \Phi)$ admits a root. Morphisms $\alpha: (M, \Phi) \to (N, \Psi)$ are simply morphisms of the underlying quasi-coherent modules that are compatible with $\Phi$ and $\Psi$, i.e.\ the following diagram is commutative: \[\begin{xy} \xymatrix{F^\ast M \ar[r]^\Phi \ar[d]^{F^\ast \alpha} & M \ar[d]^\alpha\\ F^\ast N \ar[r]^\Psi & N}\end{xy}\]}
\end{enumerate}
\end{Def}

Unit $F$-modules are special cases of left $\mathcal{D}_X$-modules and come equipped with functors $f_+$ and $f^!$ for any morphism $f:X \to Y$ between smooth schemes (see \cite[\S \S 2 and 3]{emertonkisinrhunitfcrys}. The functor $f_+$ corresponds to the pushforward on the underlying $\mathcal{D}_X$-module while $f^!$ corresponds to $f^\ast$ on the underyling quasi-coherent $\mathcal{O}_X$-module (see \cite[2.3.3]{emertonkisinrhunitfcrys}). 

We can use $f^!$ to define a notion of unit $F$-modules for embeddable schemes:

\begin{Def}
If $i: X \to X'$ is a closed immersion with $X'$ smooth and $X$ arbitrary, $j: U \to X'$ the complement, then we define the \emph{category of  unit $F$-modules on $X$} (again denoted by \emph{unit $F(X)$}) as the full subcategory of unit $F$-modules on $X'$ such that for each object $M$ we have $j^!M= 0$. It is shown in \cite[Corollary 4.11]{schedlmeiercartierpervers} that this is independent of the embedding.
\end{Def}

Similar as in the case of Cartier crystals it makes more sense to take the derived category of a larger class of $F$-modules and then restrict cohomology. Namely, an \emph{$F$-module} is a pair $(M, \Phi)$, where $M$ is an $\mathcal{O}_X$ module and $\Phi: F^\ast M \to M$ an $\mathcal{O}_X$-linear map.

\begin{Def}
Let $X$ be an embeddable scheme.
\begin{enumerate}[(a)]
\item{We denote the bounded derived category of $F$-modules with cohomology in unit $F$-modules by $D^b_{unit\, F}(X)$.}
\item{We denote the bounded derived category of unit $F$-modules by $D^b(unit\, F(X))$}
\end{enumerate}
\end{Def}

\begin{Prop}
\label{unitFDerivedCatEquivalent}
Let $X$ be an $F$-finite embeddable scheme. Then the natural functor $D^b(unit\, F(X)) \to D^b_{unit\, F}(X)$ is an equivalence.
\end{Prop}
\begin{proof}
See \cite[\S \S 17]{emertonkisinrhunitfcrys} for the smooth case and \cite{schedlmeiercartierpervers} for the embeddable case.
\end{proof}

\begin{Bla}
\label{SigmaConstruction}
Let $X$ be a smooth $F$-finite scheme. We construct a functor $\Sigma_X$ from coherent $\mathcal{O}_X\langle \kappa \rangle$-moduless on $X$ to unit $F$-modules on $X$ as follows. Let $(M, \kappa)$ be a coherent Cartier module. By duality for a finite morphism the Cartier structure $\kappa: F_\ast M \to M$ corresponds to a morphism $C: M \to F^!M$. Taking the direct limit over the iterates of this map we obtain $\mathcal{M} = \colim_{e \geq 0} F^{e!} M$. Moreover, $C$ induces an isomorphism $\Psi: \mathcal{M} \to F^! \mathcal{M}$. Tensoring with $\omega_X^{-1}$ and using the fact that one has an isomorphism $F^! \mathcal{M} \otimes \omega^{-1} \cong F^\ast \mathcal{M}$ one obtains a unit $F$-module structure on $\mathcal{M} \otimes \omega_X^{-1}$ via $\Psi^{-1} \otimes \id_{\omega_X^{-1}}$. This yields a functor $\Sigma_X$ from coherent Cartier modules on $X$  to unit $F$ modules on $X$ for any $F$-finite smooth scheme $X$.
\end{Bla}

\begin{Theo}
\label{CartierunitFequivalence}
Let $X$ be an $F$-finite smooth scheme. The functor $\Sigma_X$ constructed in \ref{SigmaConstruction} induces an equivalence between $\mathcal{O}_X\langle \kappa \rangle$-crystals on $X$ and unit $F$-modules on $X$. Moreover, $\Sigma_X$ extends to give an equivalence of derived categories $\Sigma_X: D^b_{crys}(X) \to D^b_{unit\, F}(X)$. This functor can be extended to induce a derived equivalence $\Sigma_X: D^b_{crys}(X) \to D^b_{unit\, F}(X)$ for any $F$-finite embeddable scheme $X$.
\end{Theo}
\begin{proof}
See \cite[Theorem 5.15]{blickleboecklecartierfiniteness} for the fact that $\Sigma_X$ descends to crystals and induces an equivalence of abelian categories (note that the isomorphism $\omega_X \cong F^!\omega_X$ is induced by the choice of an isomorphism $k \to F^! k$ and that \lq\lq{}Cartier modules\rq\rq{} should read \lq\lq{}Cartier crystals\rq\rq{}).

This generalizes to a derived equivalence of embeddable schemes by \cite[Theorem 5.12]{schedlmeiercartierpervers}. More precisely, if $Z$ is a scheme and $i: Z \to X$ a closed immersion with $X$ smooth then the equivalence is given by $\Sigma_X \circ i_\ast$, where $\Sigma_X$ is the derived equivalence in the smooth case. This does not depend on the chosen embedding by \cite[Proposition 5.4]{schedlmeiercartierpervers}.
\end{proof}

Let us denote by $D^b_c(X)$ the bounded derived category of $\mathbb{F}_p$-sheaves with constructible cohomology. Then in \cite{emertonkisinrhunitfcrys} a functor \[\Sol: D^b_{unit \, F}(X) \to D^b_c(X)\] is constructed. One has:

\begin{Theo}[Riemann-Hilbert correspondence]
\label{unitFConstructibleAntiequivalence}
Let $X$ be a smooth scheme. The functor $\Sol: D^b_{unit \, F}(X) \to D^b_c(X)$ is an anti-equivalence of categories. Moreover one has:
\begin{enumerate}[(a)]
\item{If $f: X \to Y$ is a morphism between smooth schemes, then there is a natural isomorphism $\Sol \circ f^! \cong f^{-1} \circ \Sol$.}
\item{If $f: X \to Y$ is a morphism between smooth schemes, and $f = gh$ with $g$ an immersion and $h$ smooth and proper, then there is a natural isomorphism $\Sol \circ f_+ \cong f_! \circ \Sol$.} 
\end{enumerate}
Finally, if $X$ is only embeddable, then $\Sol$ extends to an equivalence of categories $\Sol: D^b_{unit\, F}(X) \to D^b_c(X)$. Moreover, $f_! \circ \Sol \cong \Sol \circ f_+$ and $\Sol \circ f^{!} \cong f^{-1} \circ \Sol$ for a locally closed immersion $f$.
\end{Theo}
\begin{proof}
For the fact that $\Sol$ is an equivalence as well as (a) and (b) see \cite[Theorem 11.3]{emertonkisinrhunitfcrys}. For the case of embeddable schemes and compatibility with locally closed immersions  see \cite[Theorem 5.12]{schedlmeiercartierpervers}.
\end{proof}

We will show in Section \ref{EquivalenceShriekPullbackCompatible} that $\Sol \circ f^! \cong f^{-1} \circ \Sol$ for a smooth morphism $f$ between embeddable schemes.

In order to have a notion of relative dimension for an embeddable scheme we have to assume that the irreducible components of the schemes we consider coincide with connected components. This is in particular satisfied if the scheme is normal or irreducible. As mentioned in our conventions we will therefore impose this condition on the schemes we consider.

\begin{Def}
For each $X \in Sch_{\text{emb}, k}$ we define the \emph{trivial $t$-structure} on $D^b_?(X)$ by taking cohomology at the dimension of the corresponding irreducible components, where $? \in \{crys, unit\, F\}$.
\end{Def}

If $i: X \to Y$ is a closed immersion of irreducible schemes of relative dimension $n$, then in Cartier crystals we will denote by $i^!$ the $n$th derivative of the functor $\bar{i}^\ast \Hom(i_\ast \mathcal{O}_X, \bullet)$. For a smooth morphism $f: X \to Y$ we (re)define $f^!$ as $f^!$ shifted by the relative dimension. Then $f^!$ induces an exact functor between the hearts of the $t$-structures. Similarly, if $f:X \to Y$ is an lci morphism, then $f$ also induces an exact functor between hearts. 

With this in mind one has that the equivalence $\Sigma$ from Cartier crystals to unit $F$-modules commutes with $j^!$ and $j_\ast$ corresponds to $j_+$ for a locally closed immersion (see \cite[Theorem 5.12]{schedlmeiercartierpervers}). In particular, under the anti-equivalence $\Sol \circ \Sigma$ from Cartier crystals to perverse constructible $\mathbb{F}_p$-sheaves $j^!$ corresponds to $j^{-1}$ and $j_+$ to $j_!$ for a locally closed immersion $j$.

\begin{Ko}
Under the anti-equivalence $\Sol$ of Theorem \ref{unitFConstructibleAntiequivalence} the trivial $t$-structure on $D^b_{unit\, F}(X)$ corresponds to the middle perversity. In particular, if $X$ is embeddable, then the category of unit $F$-modules on $X$ is anti-equivalent to perverse constructible $\mathbb{F}_p$- sheaves on the \'etale site $X_{\'et}$.
\end{Ko}
\begin{proof}
See \cite[Theorem 11.5.4]{emertonkisinrhunitfcrys}.
\end{proof}

Combining Theorem \ref{CartierunitFequivalence} with Theorem \ref{unitFConstructibleAntiequivalence} we obtain
\begin{Ko}
\label{CrysPervEquivalence}
Let $X$ be an $F$-finite embeddable scheme. We have an anti-equivalence $\Sol \circ \Sigma_X\colon D^b_{crys}(X) \to D^b_c(X)$ where for a locally closed immersion $j$ we have $\Sol \circ \Sigma \circ j^! \cong j^{-1} \circ \Sol \circ \Sigma$ and $\Sol \circ \Sigma \circ j_\ast \cong j_+ \circ \Sol \circ \Sigma$. Moreover, this anti-equivalence maps the trivial $t$-structure to the middle perversity and therefore induces an equivalence of Cartier crystals with perverse constructible $\mathbb{F}_p$-sheaves on the \'etale site $X_{\'et}$.
\end{Ko}

\section{The functor $f^!$ for a smooth morphism $f$}
\label{EquivalenceShriekPullbackCompatible}

The results in this section should be of independent interest. We prove that the equivalence between Cartier crystals and unit $F$-modules for embeddable schemes (as in Theorem \ref{CartierunitFequivalence}) commutes with $f^!$ for smooth morphisms $f$, i.e.\ $\Sigma \circ f^! \cong f^! \circ \Sigma$. This is accomplished in Theorems \ref{SmoothEquivalenceSmooth} and \ref{ShriekSmoothEmbeddableCommute}. We then proceed to show that under the anti-equivalence between Cartier crystals and perverse constructible sheaves $f^!$ corresponds to $f^{-1}$ for smooth morphisms between embeddable schemes (Corollary \ref{CrystoPervSmoothEmbeddable}).

Recall that we have a functor $\Sigma$ from coherent Cartier modules to unit $F$-modules  that induces an equivalence when passing to crystals (Theorem \ref{CartierunitFequivalence}). The following lemma explicitly describes the adjoint of the unit $F$ structural map in terms of the Cartier module via the functor $\Sigma$.
\begin{Le}
\label{CartiertoUnitFExplicit}
Let $R$ be smooth over some $F$-finite field and fix an isomorphism $C: \omega_R \to F^! \omega_R$ with adjoint $\kappa: F_\ast \omega_R \to \omega_R$. Let $(M, \kappa_M)$ be a Cartier module and write $(\Hom(\omega_R, \mathcal{M}), \Psi)$ for the corresponding unit $F$-module.
Then the adjoint map to $\Psi: F^\ast \Hom(\omega_R, \mathcal{M}) \to \Hom(\omega_R, \mathcal{M})$, namely \[\Phi: \Hom(\omega_R, \mathcal{M}) \to F_\ast\Hom(\omega_R,\mathcal{M})\] is given as the following composition of maps
\[ \Hom(\omega_R, \mathcal{M}) \longrightarrow \Hom(F_\ast \omega_R, \mathcal{M}) \longrightarrow F_\ast \Hom(\omega_R, F^! \mathcal{M}) \longrightarrow F_\ast \Hom(\omega_R, \mathcal{M})\]
where the arrows are induced by $\kappa: F_\ast \omega_R \to \omega_R$, adjunction of $F_\ast$ and $F^!$ and the inverse of $C_{\mathcal{M}}: \mathcal{M} \to F^! \mathcal{M}$ respectively.  Here $C_{\mathcal{M}}$ is obtained by taking the direct limit over the maps $M \to F^{e!} M$ which are adjoint to $\kappa^e: F_\ast^e M \to M$.
\end{Le}

\begin{proof}
We may identify $\Hom(\omega_R, \bullet)$ with $\bullet \otimes \omega_R^{-1}$. Recall that the unit $F$ structural map $F^\ast ( \mathcal{M} \otimes \omega_R^{-1}) \to \mathcal{M} \otimes \omega_R^{-1}$ is induced from $C_{\mathcal{M}}^{-1} \otimes \id_{\omega_R^{-1}}$ by the isomorphism $\alpha: F^\ast (\mathcal{M} \otimes \omega_R^{-1} ) \to F^! \mathcal{M} \otimes \omega_R^{-1}$ (cf. \cite[Corollary 5.8]{blickleboecklecartierfiniteness}).

In particular, it suffices to verify that the adjoint of $F_\ast(C_{\mathcal{M}} \otimes \id_{\omega_R^{-1}}) \circ \Phi$ is $\alpha$. Note that the adjoint of $F_\ast(C_{\mathcal{M}} \otimes \id_{\omega_R^{-1}}) \circ \Phi$ is given by \[F^\ast \Hom(\omega_R, \mathcal{M}) \longrightarrow \Hom(\omega_R, F^! \mathcal{M}),\quad t \otimes \varphi \longmapsto t [ds \mapsto [r \mapsto \varphi(\kappa(r ds))]].\]

The map $\alpha$ comes about as follows: Tensoring the isomorphism $F^! R \otimes F^\ast \omega_R \to F^! \omega$ (apply \cite[Lemma 5.7]{blickleboecklecartierfiniteness} with $M = \omega_R$) with $F^\ast \omega_R^{-1}$ we get an isomorphism \[F^! R \longrightarrow F^! \omega_R \otimes F^\ast \omega_R^{-1}.\] Let us denote its inverse by $\lambda$. By \cite[Lemma 5.7]{blickleboecklecartierfiniteness} we have an isomorphism $F^! R \otimes F^\ast \mathcal{M} \to F^! \mathcal{M}$. Then $\lambda$ induces \begin{equation}F^! \omega_R \otimes F^\ast \omega_R^{-1} \otimes F^\ast \mathcal{M} = F^! \omega_R \otimes F^\ast \Hom(\omega_R, \mathcal{M}) \to F^! \mathcal{M}. \label{Equation1} \end{equation} Now one finally identifies $F^! \omega_R$ with $\omega_R$ via the fixed isomorphism $C^{-1}$ and tensors both sides with $\omega^{-1}_R$.

First, we claim that the morphism \eqref{Equation1} is given by the natural map $\varphi \otimes \psi \mapsto \psi \circ \varphi$. In order to verify this it suffices to show that the composition of this natural map with $F^! R \otimes F^\ast \mathcal{M} \to F^! \omega_R \otimes F^\ast \omega^{-1}_R \otimes F^\ast \mathcal{M}$ coincides with the given map $F^! R \otimes F^\ast \mathcal{M} \to F^! \mathcal{M}$. This is a local issue so that we may assume that $\omega_R$ is a free $R$-module generated by $ds$. Then the map \[F^! R \otimes  F^\ast \mathcal{M} \to F^! \omega_R \otimes F^\ast \omega^{-1}_R \otimes F^\ast \mathcal{M} \to F^! \omega_R \otimes F^\ast \Hom(\omega_R, \mathcal{M})\] is given by sending $\varphi \otimes t \otimes m$ to $[r \mapsto \varphi(tr) ds] \otimes 1 \otimes [rds \mapsto (ds)^\vee (r ds) m ]$ and one readily checks that composition with the natural map yields the claimed isomorphism.

Now by the above the map $\omega_R \otimes F^\ast \Hom(\omega_R, \mathcal{M}) \to F^! \mathcal{M}$ is given by $ds \otimes \psi \mapsto [r \mapsto \psi(\kappa(rds))]$. Tensoring with $\omega^{-1}_R$ and making the identification $\omega^{-1}_R \otimes F^! \mathcal{M} = \Hom(\omega_R, F^! \mathcal{M})$ finally yields the map \[ds^\vee \otimes ds \otimes \psi \mapsto [dt \mapsto [r \mapsto \psi(\kappa(r ds^\vee(dt) ds))]].\] Since $dt = u ds$ and then $ds^\vee(dt) = u$ this coincides with the adjoint of $F_\ast(C_{\mathcal{M}} \otimes \id_{\omega_R^{-1}}) \circ \Phi$ as described above.
\end{proof}

\begin{Bem}
In practice if $R$ is (essentially) of finite type over an $F$-finite field $k$, then one fixes once and for all an isomorphism $k \to F^! k$. If $f: \Spec R \to \Spec k$ is the structural map, then $f^!$ induces an isomorphism $\omega_R \to F^! \omega_R$. If $k$ is not perfect, then there is no canonical choice for the isomorphism $k \to F^! k$.
\end{Bem}

\begin{Def}
Let $f: X \to Y$ be a smooth morphism of affine schemes. We call $f$ \emph{standard smooth} if $f$ factors as $g \circ \varphi$, where $g: \mathbb{A}^n_Y \to Y$ is the structural map and $\varphi: X \to \mathbb{A}_Y^n$ is \'etale.
\end{Def}

\begin{Theo}
\label{SmoothEquivalenceSmooth}
Let $f: X \to Y$ be a smooth morphism of smooth schemes. Then if $\Sigma_?$ denotes the equivalence $Crys(?) \to unit\,F(?)$ of Theorem \ref{CartierunitFequivalence}, one has $\Sigma_X \circ f^! \cong f^! \circ \Sigma_Y$. 
\end{Theo}
\begin{proof}
Recall that the functor $f^!$ on unit $F$-modules is simply $f^\ast$ on underlying $\mathcal{O}_?$-modules (and we will use this notation in the proof). Observe that we have a natural isomorphism $\beta$ of underlying $\mathcal{O}_X$-modules: 
\begin{align*}f^\ast \Sigma_Y(M) =  f^\ast (\colim_{e \geq 0} {F^e_Y}^! (M) \otimes \omega_Y^{-1}) &\cong f^\ast(\colim_{e \geq 0} {F^e_Y}^!M)\otimes \omega_{X/Y} \otimes \omega_X^{-1}\\ &\cong f^! (\colim_{e \geq 0} {F^e_Y}^! M) \otimes \omega_X^{-1} \\& \cong (\colim_{e \geq 0} f^! {F^e_Y}^! M) \otimes \omega_X^{-1} \\ &\cong  (\colim_{e \geq 0} {F^e_X}^! f^! M) \otimes \omega_X^{-1} = \Sigma_X(f^!M), \end{align*}
where for the second and third isomorphism we use that $f$ is smooth so that $f^! M = f^\ast M \otimes \omega_{X/Y}$.
It remains to verify that this interchanges unit $F$ structures, i.e.\ we have to verify that the following diagram is commutative: \begin{equation} \label{DiagramCommute} \begin{xy} \xymatrix{f^\ast \Sigma(M) \ar[r]^\beta & \Sigma(f^!M) \\ F^\ast f^\ast \Sigma(M) \ar[r]^{F^\ast \beta} \ar[u] & F^\ast \Sigma(f^!M) \ar[u]} \end{xy} \end{equation}
Here the vertical maps are the unit $F$-module structures on $f^\ast \Sigma(M)$ and $\Sigma(f^!M)$ respectively. 
The commutativity of (\ref{DiagramCommute}) is a local statement so that we may assume that $f = g \circ h$ is standard smooth, with $h: \Spec S \to \Spec \mathbb{A}^n_R$ \'etale and $g: \mathbb{A}^n_R \to \Spec R$. We treat these cases separately, i.e.\ we have to look at the case of an \'etale morphism $\Spec S \to \Spec R$ and a smooth morphism $\mathbb{A}^n_R \to \Spec R$. Fix an $R$ Cartier module $(M, \kappa)$ and denote $ \colim_{e \geq 0} F^{e!} M$ by $\mathcal{M}$. We denote the adjoint structural map of $M$ by $C$. From here on out we will omit the subscript on $\Sigma$.

Now $\Sigma(M) = \Hom(\omega_R, \mathcal{M})$ comes equipped with a unit $F$-structure which admits an adjoint $\Phi: \Hom(\omega_R, \mathcal{M}) \to F_\ast \Hom(\omega_R, \mathcal{M})$.
If $f$ is any smooth morphism, then a small computation shows that the adjoint structural map of the unit $F$-module $f^\ast \Sigma(M) = f^\ast \Hom(\omega_R, \mathcal{M})$ is given by 
\begin{equation} \label{PsiDef}\Psi: f^\ast \Hom(\omega_R, \mathcal{M}) \longrightarrow F_\ast f^\ast \Hom(\omega_R, \mathcal{M}), \quad s \otimes \varphi \longmapsto s^p \otimes \Phi(\varphi).\end{equation}

Next we denote the adjoint structural map of the unit $F$-module $\Sigma(f^!M)$ by $\Xi$. Instead of verifying the commutativity of diagram (\ref{DiagramCommute}) we may also verify that the corresponding diagram with adjoint structural maps commutes. This is what we shall do. Moreover, we may also compose the claimed equality with an isomorphism and will therefore show that
\begin{equation} \label{Whatweshow} \Hom(\omega_S, -)( F_\ast \Lambda) \circ \Xi \circ \beta = \Hom(\omega_S, -)( F_\ast \Lambda) \circ F_\ast \beta \circ \Psi, \end{equation}
 where $\Lambda: f^! \mathcal{M} \to F^! f^! \mathcal{M}$ is the adjoint structure map of $f^! \mathcal{M}$. For the convenience of the reader let us draw the  diagram corresponding to (\ref{Whatweshow}):

\[\begin{xy} \xymatrix{f^\ast \Hom(\omega_R, \mathcal{M}) \ar[r]^{\beta} \ar[dr]^{\Psi} & \Hom(\omega_S, f^! \mathcal{M}) \ar[r]^{\Xi} & F_\ast \Hom(\omega_S, f^! \mathcal{M}) \ar[r]^{ F_\ast \Lambda \circ -}& F_\ast \Hom(\omega_S, F_\ast f^! \mathcal{M})
\\ &F_\ast f^\ast \Hom(\omega_R, \mathcal{M})) \ar[r] & F_\ast \Hom(\omega_S, f^! \mathcal{M}) \ar[ur]_{ F_\ast \Lambda \circ -}} \end{xy} \]

Recall that Lemma \ref{CartiertoUnitFExplicit} yields a description of $\Xi$ and, together with (\ref{PsiDef}), also a description of $\Psi$.  We denote the structural map $F_\ast f^! M \to f^!M$ of $f^!M$ by $\lambda$. By abuse of notation we will use the same letter for structural maps (and their adjoints) on the colimits $f^! \mathcal{M}, \mathcal{M}$. The fixed Cartier structure on $\omega_R$ is denoted by $\kappa_\omega$. Applying $f^!$ to the isomorphism $\omega_R \to F^! \omega_R$ and then using adjunction this induces a Cartier structure $\kappa_{\omega_S}$ on $\omega_S$.

Finally, we will verify (\ref{Whatweshow}) by evaluation on elements and therefore do not need to keep track of the colimit in the natural isomorphism $\beta$. Moreover, by construction, the Cartier structure on $f^!M$ is obtained via the natural isomorphism $F^! f^! M \cong f^! F^! M$ so that it is in particular Cartier linear. We may thus also identify $f^! F^! M$ with $F^! f^!M$ when applying $\beta$.

\begin{claim}
\label{SmoothEquivalenceSmoothClaim1}
If  $f: \Spec S \to \Spec R$ is \'etale, then (\ref{Whatweshow}) holds.
\end{claim}
\begin{proof}[Proof of claim.]
Recall (Theorem \ref{CartierShriek}) that $f^! = f^\ast$ and that one has an isomorphism $\alpha: S \otimes_R F_\ast M \to F_\ast(S \otimes_R M), s \otimes m \mapsto s^p \otimes m$. In particular, any $s \in S$ can be written as $\sum_i r_i s_i^p$ with $r_i \in R$. The Cartier structure on $f^! M$ is given by $\lambda: (\id \otimes \kappa) \circ \alpha^{-1}$. Since both maps are clearly additive we may restrict our attention to tensors of the form $s^p \otimes \bullet$.

Note that $\beta$ is simply the natural isomorphism $f^\ast \Hom(\omega_R, \mathcal{M}) \to \Hom(\omega_S, f^\ast \mathcal{M})$ (since $\omega_S = f^\ast \omega_R$).
We then have that $\Hom(\omega_S, -)( F_\ast \Lambda) \circ F_\ast \beta\circ \Psi(s \otimes \varphi)(t^p \otimes w)$ is given by \[u^p \longmapsto \lambda(u^p s^p t^p \otimes C^{-1} ([r \mapsto \varphi(\kappa_\omega(rw))])).\] Using the fact that $\lambda = (\id \otimes \kappa) \circ \alpha^{-1}$ we obtain \[u^p \mapsto \id \otimes \kappa (ust \otimes C^{-1}([r \mapsto \varphi(\kappa_\omega(rw))])) = u^p \mapsto ust \otimes \varphi(\kappa_\omega(w)).\]

On the other hand, using Lemma \ref{CartiertoUnitFExplicit}, $\Hom(\omega_S, -)( F_\ast \Lambda) \circ \Xi \circ \beta$ is given as the following composition of maps
\[\begin{xy}\xymatrix{f^\ast \Hom(\omega_R, \mathcal{M}) \ar[r]^{\beta}& \Hom(\omega_S, f^\ast \mathcal{M}) \ar[r]^{\kappa_{\omega_S}} & \Hom(F_\ast \omega_S, f^\ast \mathcal{M}) \ar[r]^{\text{adj.}} & F_\ast \Hom(\omega_S, F^! f^\ast \mathcal{M}) } \end{xy}\]
An element $s \otimes \varphi$ is then mapped via $\beta$ to $t^p \otimes \omega \mapsto st^p \otimes \varphi(\omega)$. Applying $\kappa_{\omega_S}$ and then adjunction for finite maps we obtain \[t^p \otimes \omega \longmapsto [u^p \mapsto stu \otimes \varphi(\kappa_{\omega}(\omega))]\] as claimed.
\end{proof}

\begin{claim}
\label{SmoothEquivalenceSmoothClaim2}
If $f: \mathbb{A}^n_R \to \Spec R$ is the structural map, then (\ref{Whatweshow}) holds.
\end{claim}
\begin{proof}[Proof of claim.]
We may assume that $n =1$ and denote the coordinate by $x$. Moreover, we may assume that $\Omega^1_{R/k}$ is free with basis $dy_1, \ldots, dy_i$. We denote the corresponding basis on $\omega_R$ by $\delta$. Then $f^! M$ is given by $f^\ast M \otimes \omega_f$, where $\omega_f = R[x] dx$ is a free $R[x]$-module of rank $1$. One has a natural isomorphism of Cartier modules $f^! \omega_R = \omega_f \otimes f^\ast \omega_R \cong \omega_{R[x]}$ given by $\delta \otimes 1 \otimes dx \mapsto \delta \wedge dx$. We have \begin{equation}\label{BetaSmooth}\begin{split} \beta: f^\ast \Hom(\omega_R, \mathcal{M}) &\longrightarrow \Hom(\omega_{R[x]}, f^!\mathcal{M}) \\ x^a \otimes \varphi & \longmapsto [ \delta \wedge dx \mapsto \varphi(\delta) \otimes 1 \otimes x^a dx]\end{split}\end{equation}
The induced Cartier structure on $f^! \mathcal{M}$ is given by $rx^n dx \otimes m \mapsto x^{\frac{n+1}{p} -1} \otimes \kappa(rm)$ (see \cite[Lemma 4.1]{staeblerunitftestmoduln})\footnote{with the usual convention that $r \in \mathbb{Q} \setminus \mathbb{Z}$ means that  $x^{r} = 0$}.

Let us compute the right-hand side of (\ref{Whatweshow}). Since our maps are additive we may restrict our attention to tensors of the form $x^a \otimes \varphi$. Recall that by (\ref{PsiDef}) the map $\Psi: f^\ast \Hom(\omega_R, \mathcal{M}) \to F_\ast f^\ast \Hom_R(\omega_R, \mathcal{M})$ is given by $x^a \otimes \varphi \mapsto x^{ap} \otimes \Phi(\varphi)$. Now, using Lemma \ref{CartiertoUnitFExplicit} we further check that $\Psi$ is given by 
\[x^a \otimes \varphi \longmapsto x^{ap} \otimes [r \delta \mapsto C^{-1}[s \mapsto \varphi(\kappa_\omega(rs \delta))]]. \]
Next, we compose with $F_\ast \beta$ (see (\ref{BetaSmooth}) above) and obtain that $F_\ast \beta \circ \Psi$ is of the form
\[x^a \otimes \varphi \longmapsto [ r\delta \wedge x^b dx \mapsto [C^{-1}[s \mapsto \varphi(\kappa_\omega(rs\delta)] \otimes 1 \otimes x^{ap + b} dx]. \] Finally composing with $\Hom(\omega_S, -)( F_\ast \Lambda)$ we get that $\Hom(\omega_S, -)( F_\ast \Lambda) \circ F_\ast \beta \circ \Psi$ is given by
\begin{equation} \label{SmoothRighthandside} x^a \otimes \varphi \longmapsto [r\delta \wedge x^bdx \mapsto [sx^c \mapsto \varphi(\kappa_\omega(rs\delta)) \otimes 1 \otimes x^a x^{\frac{c+b+1}{p} -1}dx]]. \end{equation}
Now we compute the left-hand side of (\ref{Whatweshow}). First we apply $\beta$ to $x^a \otimes \varphi$ which we can read off from (\ref{BetaSmooth}). Next we apply $\Hom(\omega_S, -)( F_\ast \Lambda) \circ \Xi$, using Lemma \ref{CartiertoUnitFExplicit} we get that $\Hom(\omega_S, -)( F_\ast \Lambda) \circ \Xi \circ \beta$ is of the form
\begin{equation} \label{SmoothLefthandside} x^a \otimes \varphi \longmapsto [r \delta \wedge x^b dx \mapsto [sx^c \mapsto \varphi(\kappa_\omega(rs\delta)) \otimes 1 \otimes x^a x^{\frac{b+ c +1}{p} -1}dx]]. \end{equation} Note here that $\Xi$ consists of three steps: precompose with $\kappa_{\omega_{R[x]}}$, apply adjunction and then $\Hom(\omega_S, -)(F_\ast \Lambda^{-1})$. Since we precompose with $\kappa_{\omega_{R[x]}}$ we get \[x^a \kappa_{\omega_{R[x]}}( \varphi(r \delta) \otimes 1 \otimes x^b dx ) =  \kappa_{\omega}(\varphi(r\delta)) \otimes x^{a} x^{\frac{b+1}{p} -1} dx \otimes 1\] and in order to obtain (\ref{SmoothLefthandside}) we just have to write out the adjunction.

We see that (\ref{SmoothRighthandside}) and (\ref{SmoothLefthandside}) coincide which proves the claim.
\end{proof}
Combining Claims \ref{SmoothEquivalenceSmoothClaim1} and \ref{SmoothEquivalenceSmoothClaim2} shows the theorem.
\end{proof}

\begin{Ko}
\label{SmoothDerivedSmoothEquivalence}
Let $f: X \to Y$ be a smooth morphism of smooth schemes. Then if $\Sigma_X$ denotes the equivalence $D^b_{crys}(X) \to D^b_{unit\, F}(X)$ (as discussed in Theorem \ref{CartierunitFequivalence}) and similarly for $\Sigma_Y$ one has $\Sigma_X \circ f^! \cong f^! \circ \Sigma_Y$.
\end{Ko}
\begin{proof}
As $f^!$ is exact in both cases this is an immediate consequence of Theorem \ref{SmoothEquivalenceSmooth}, Proposition \ref{CrystalsDerivedCatEquivalent} and Proposition \ref{unitFDerivedCatEquivalent}.
\end{proof}

In order to achieve a result similar to Corollary \ref{SmoothDerivedSmoothEquivalence} for embeddable schemes and $f^!$ on $Crys$ and $f^{-1}$ on perverse constructible sheaves, i.e.\ $f^{-1} \circ \Sol \circ \Sigma \cong \Sol \circ \Sigma \circ f^!$ we need to verify one more compatibility of the shriek functor. That is, we need to verify that if $b: X \to Y$ is a closed immersion and $g: Y \to Z$ and $gb$ are smooth then the isomorphism $(gb)^! \cong b^! g^!$ is compatible with Cartier structures. In particular,  $b$ has to be a locally complete intersection (cf.\ \cite[Th\'eor\`eme II.4.10]{SGA1}) so that if $n$ is its relative dimension, then $b^! M = \bar{b}^\ast R^n\Hom(b_\ast \mathcal{O}_X, M)$ for a Cartier module or crystal $M$.

\begin{Le}
\label{ShriekCompatibilityFiniteSmooth}
Let $b: X \to Y$ and $g: Y \to Z$ be morphisms of $F$-finite schemes with $b$ a closed immersion and $g$ and $f = gb$ smooth. Then one has an isomorphism of Cartier modules (or crystals) $b^! g^! M \cong f^!M$.
\end{Le}
\begin{proof}
By \cite[Proposition III.8.4]{hartshorneresidues} there is an underlying isomorphism of coherent sheaves. It remains to verify that this is compatible with Cartier structures which is local. Hence, we may assume that $b$ is a complete intersection and $X = \Spec S/I$, $Y = \Spec S$, $Z = \Spec R$. Further we may assume that $g$ and $f= gb$ are standard smooth with factorizations $g  = s \circ \varphi$ and $f =t \circ \psi$. Moreover, if $x_1, \ldots, x_n$ are coordinates for $g$, then we may assume that the closed immersion $b$ is given by modding out $x_{m+1}, \ldots, x_n$ (cf. \cite[Th\'eor\`eme II.4.10]{SGA1}).

We therefore obtain the following commutative diagram
\[\begin{xy}
\xymatrix{ \Spec S/I \ar[r]^b \ar[d]^\psi & \Spec S \ar[d]^\varphi \\
\mathbb{A}^m_R \ar[r]^i \ar[dr]^t & \mathbb{A}^n_R \ar[d]^s\\
& \Spec R }
\end{xy}\]
where the square is Cartesian. We will verify that $t^! \cong i^! s^!$ and $\psi^! i^! \cong b^! \varphi^!$ are isomorphisms of Cartier modules. Then using \cite[Lemma 4.4]{staeblerunitftestmoduln} we obtain $(t \psi)^! \cong \psi^! t^! \cong \psi^! i^! s^! \cong b^! \varphi^! s^! \cong b^! (s \varphi)^!$. 

By \cite[Example 3.3.12]{blickleboecklecartiercrystals} the Cartier structure on $b^! N = N/IN$ is given by $\kappa \cdot (x_{m+1} \cdots x_n)^{p-1}$, where $\kappa$ is the Cartier structure on $N$ and $I = (x_{m+1}, \ldots, x_n)$. It is now straightforward to verify that $t^! \cong i^! s^!$ is compatible with Cartier structures. Next, we verify that $\psi^! i^! \cong b^! \varphi^!$ is compatible with Cartier structures.

So let $(M, \kappa_M)$ be a Cartier module.
We have $b^! \varphi^! M = M \otimes_{R[x_1, \cdots, x_n]} S \otimes_S S/I$, where the Cartier structure is given by \[m \otimes 1 \otimes s^p \mapsto \kappa(m \otimes 1 \otimes x_{m+1}^{p-1} \cdots x_n^{p-1} s^p) = \kappa_M(x_{m+1}^{p-1} \cdots x_n^{p-1} m) \otimes 1 \otimes s.\] On the other hand, the Cartier structure on \[\psi^! i^! M = M/(x_{m+1}, \ldots, x_n)M \otimes_{R[x_1, \ldots, x_m]} S/IS\] is given by \[m + (x_{m+1}, \ldots, x_m)M \otimes s^p \mapsto \kappa_M(x_{m+1}^{p-1} \cdots x_n^{p-1} m) + (x_{m+1}, \ldots, x_m)M \otimes s.\] These Cartier structures are clearly interchanged by  the natural isomorphism $b^! \varphi^! M \to \psi^! i^! M$.
\end{proof}

\begin{Le}
\label{SomeLemma}
Let $f: X \to Y$ be a smooth morphism of embeddable schemes. Then there is a commutative diagram
\[\begin{xy}
\xymatrix{X \ar[r]^-{a} \ar[d]^f & \tilde{X} \ar[d]^g\\ Y  \ar[r]^i & \tilde{Y}}
\end{xy}\]
with $g$ smooth and $a, i$ closed immersions and $\tilde{X}, \tilde{Y}$ smooth. For any such diagram one has a natural isomorphism $a_\ast f^! \cong a_\ast a^! g^! i_\ast$ in the category of Cartier crystals or modules.
\end{Le}
\begin{proof}
Choose smooth embeddings $i: Y \to \tilde{Y}$ and $c: X \to X'$. The morphisms $c$ and $i \circ f$ induce a morphism $a: X \to X' \times_k \tilde{Y} = \tilde{X}$. Denote by $pr_1$ the projection $\tilde{X} \to X'$. Since $c$ is affine and $c = pr_1 \circ a$ we conclude from \cite[II.1.6.2 (v)]{EGAII} that $a$ is affine. Since $c$ is surjective on sections we conclude the same for $a$ which implies that it is a closed immersion. 
We may take $g$ to be the second projection $pr_2: \tilde{X} = X' \times_k \tilde{Y} \to \tilde{Y}$ which is smooth since it is a base change of $X' \to \Spec k$.

We come to the claimed isomorphism of functors. We have the following commutative diagram
\[\begin{xy}
\xymatrix{X \ar@/^/[drr]^-{a} \ar[dr]^b \ar@/_/[ddr]_f &&\\& \tilde{X}\times_{\tilde{Y}} Y \ar[r]^-{i'} \ar[d]^{g'} & \tilde{X} \ar[d]^g\\
&Y \ar[r]^i & \tilde{Y}}
\end{xy}\]
where for the pullback diagram one has $g^! i_\ast \cong i'_\ast g'^!$ by \cite[Lemma 4.5]{staeblerunitftestmoduln}. Note that $b$ is a closed immersion, so that Lemma \ref{ShriekCompatibilityFiniteSmooth} yields that $b^! g'^! \cong f^!$. Using these facts one computes \[a_\ast a^! g^! i_\ast \cong a_\ast a^! i'_\ast g'^! \cong a_\ast b^! i'^! i'_\ast g'^! \cong a_\ast b^! g'^! \cong a_\ast f^!\] which is the asserted isomorphism.
\end{proof}

\begin{Bem}
Note that we have not used the coherence assumption of $M$ in Lemma \ref{ShriekCompatibilityFiniteSmooth} anywhere. In particular, the assertion of Lemma \ref{ShriekCompatibilityFiniteSmooth} continues to hold in the category of quasi-coherent Cartier modules or the category $QCrys(X)$. Deriving we also obtain corresponding statements for $D^b(QCrys(X))$ and its full subcategory $D^b_{crys}(X)$. 
\end{Bem}

\begin{Prop}
\label{CartierSmoothPullbackCompatibleWithEmbeddings}
Let $f: X \to Y$ be a smooth morphism of embeddable schemes. Then there is a commutative diagram
\[\begin{xy}
\xymatrix{X \ar[r]^-{a} \ar[d]^f & \tilde{X} \ar[d]^g\\ Y  \ar[r]^i & \tilde{Y}}
\end{xy}\]
with $g$ smooth and $a, i$ closed immersions and $\tilde{X}, \tilde{Y}$ smooth. For any such diagram the composition of derived  functors $a_\ast f^!: D^b(QCrys(Y)) \to D^b(QCrys(X)) \to D^b(QCrys(\tilde{X}))$ is naturally isomorphic to $a_\ast a^! g^! i_\ast$. Likewise, we have a natural isomorphism $a_\ast f^! \cong a_\ast a^! g^! i_\ast$ if we restrict to the full subcategories $D^b_{crys}(?)$ throughout.
\end{Prop}
\begin{proof}
Lemma \ref{SomeLemma} together with the Remark above yields the corresponding statement in $D^b(QCrys)$ and hence also in $D^b_{crys}(X)$.
\end{proof}

\begin{Le}
\label{SmoothInverseImageCompatibleWithEmbeddings}
Let $f: X \to Y$ be a smooth morphism of embeddable schemes. Then there is a commutative diagram
\[\begin{xy}
\xymatrix{X \ar[r]^-{a} \ar[d]^f & \tilde{X} \ar[d]^g\\ Y  \ar[r]^i & \tilde{Y}}
\end{xy}\]
with $g$ smooth and $a, i$ closed immersions and $\tilde{X}, \tilde{Y}$ smooth. For any such diagram the composition of derived  functors $a_! f^{-1}: D^b_c(Y) \to D^b_c(X) \to D^b_c(\tilde{X})$ is naturally isomorphic to $a_! a^{-1} g^{-1} i_!$.
\end{Le}
\begin{proof}
This follows using the proper base change theorem (e.g.\ \cite[Corollary VI.2.3]{milne}) along the lines of \ref{CartierSmoothPullbackCompatibleWithEmbeddings}. Note that the base change theorem holds for arbitrary abelian torsions sheaves so that we do not need to worry about the presentation of $D^b_c(\tilde{X})$ in this case.
\end{proof}

\begin{Theo}
\label{derivedequivtransformsshriek}
For embeddable schemes $X,Y$ and a smooth morphism $f: X \to Y$ if $\Phi_X = \Sol_X \circ \Sigma_X$ denotes the chain of equivalences $D^b_{crys}(X) \to D^b_{unit\,F}(X) \to D^b_c(X)$ (as in Corollary \ref{CrysPervEquivalence}) and similarly $\Phi_Y = \Sol_Y \circ \Sigma_Y$, then $\Phi_X \circ f^! \cong f^{-1} \circ \Phi_Y$.
\end{Theo}
\begin{proof}
Note that $\Phi \circ i_\ast \cong i_! \circ \Phi$ and $\Phi \circ i^! \cong i^{-1} \circ \Phi$ for locally closed immersions by \cite[Theorem 5.12 (c)]{schedlmeiercartierpervers}. By Corollary \ref{SmoothDerivedSmoothEquivalence}, \cite[Theorem 11.3]{emertonkisinrhunitfcrys} we have $\Phi_X \circ f^! \cong f^{-1} \circ \Phi_Y$ for a morphism $f: X \to Y$ between \emph{smooth} schemes. Using the setup of Proposition \ref{CartierSmoothPullbackCompatibleWithEmbeddings} and combining it with Lemma \ref{SmoothInverseImageCompatibleWithEmbeddings} above we obtain that  $\Phi \circ (a_\ast a^! g^! b_\ast) \cong (a_! a^{-1} g^{-1} b_!) \circ \Phi$. Recall that $i^{-1} i_! \cong \id$ for any closed immersion $i: Z \to X$ in $D^b_c(Z)$ and similarly $i^! i_\ast \cong \id$ in $D^b_{crys}(Z)$. Using this observation with $i = a$ we obtain that $f^!$ corresponds to $f^{-1}$ via $\Phi$.
\end{proof}

\begin{Ko}
\label{CrystoPervSmoothEmbeddable}
For embeddable schemes $X,Y$ and a smooth morphism $f: X \to Y$ under the chain of equivalences $Crys(?) \to$ unit $F(?) \to Perv_c(?)$ (as in Corollary \ref{CrysPervEquivalence}), where $Perv_c(?)$ denotes perverse constructible sheaves on $?$ the functor $f^!$ corresponds to $f^{-1}$.
\end{Ko}
\begin{proof}
Since under the derived equivalence $D^b_{crys}(X) \to D^b_{\text{c}}(X)$ the trivial $t$-structure is sent to the perverse $t$-structure (\cite[Corollary 5.14]{schedlmeiercartierpervers}) this follows from Theorem \ref{derivedequivtransformsshriek}.
\end{proof}

Recall that $R\Gamma_X$ for a closed subscheme $X$ of a scheme $X'$ denotes the (derived) local cohomology functor and recall our convention that morphisms in embeddable schemes are $k$-linear (cf. Conventions just before Section \ref{Cartiercrystals}). See also our Conventions for the notion of relative dimension.

\begin{Def}
\label{DefGeneralPullbackunitF}
Let $f: X \to Y$ be a smooth morphism of embeddable schemes over a fixed $F$-finite base field $k$. If $i: Y \to \tilde{Y}$ and $a: X \to \tilde{X}$ are $k$-embeddings into smooth schemes such that $if = ga$ for some smooth $k$-morphism $g: \tilde{X} \to \tilde{Y}$, then for any unit $F$-module $M$ on $Y$ we define $f^! M$ as $R^n\Gamma_X g^! M$, where $n$ is the relative dimension of the embedding $a$ (see our Conventions for the notion of relative dimension). Note that, by definition, $M$ is a unit $F$-module on $\tilde{Y}$ such that $j^! M = 0$, where $j: \tilde{Y} \setminus Y \to \tilde{Y}$.
\[\begin{xy}\xymatrix{X \ar[r]^f \ar[d]^a & Y \ar[d]^i \\ \tilde{X} \ar[r]^g & \tilde{Y}}\end{xy}\]
\end{Def}

\begin{Theo}
\label{ShriekSmoothEmbeddableCommute}
With the notation of \ref{DefGeneralPullbackunitF} if $\Sigma_?$ denotes the equivalence $Crys(?) \to unit$ $F(?)$ (see Theorem \ref{CartierunitFequivalence}), then $\Sigma_X \circ f^! \cong f^! \circ \Sigma_Y$. In particular, the definition of $f^!$ for unit $F$-modules does not depend on the embedding.
\end{Theo}
\begin{proof}
Recall that by definition the equivalence $\Sigma_Y$ is given by $\Sigma_{\tilde{Y}} \circ i_\ast$ and similarly for $\Sigma_X$. Let us denote by $n_Y$ the relative dimension of $i$ and by $n_X$ the relative dimension of $a$.

Then we have we have $R^{n_Y}\Gamma_Y \cong i_\ast i^!$ in $Crys(\tilde{Y})$ by \cite[Lemma 3.2]{blicklestaeblerfunctorialtestmodules} and similarly for $Crys(\tilde{X})$. By loc.\ cit.\ and \cite[Lemma 5.7]{schedlmeiercartierpervers} we have $\Sigma_{\tilde{X}} R^{n_X} \Gamma_X \cong R^{n_X} \Gamma_X \Sigma_{\tilde{X}}$. Next note that whenever we have a morphism $g$ as in \ref{DefGeneralPullbackunitF}, then Lemma \ref{SomeLemma} shows that we have $a_\ast f^! \cong a_\ast a^! g^! i_\ast$.

Using Theorem \ref{SmoothEquivalenceSmooth} and the above statements we obtain \begin{align*} f^! \Sigma_Y &\cong R^n\Gamma_X g^! \Sigma_Y \cong R^{n_X}\Gamma_X g^! \Sigma_{\tilde{Y}} i_\ast \cong R^{n_X}\Gamma_X \Sigma_{\tilde{X}} g^! i_\ast \\&\cong \Sigma_{\tilde{X}} R^{n_X}\Gamma_X g^! i_\ast \cong \Sigma_{\tilde{X}} a_\ast a^! g^! i_\ast \cong \Sigma_{\tilde{X}}a_\ast f^! \cong \Sigma_X f^!.\end{align*}
\end{proof}

In order to obtain the corresponding statement for the derived category we need one more lemma.

\begin{Le}
\label{GammaDescription}
Let $X$ be an embeddable scheme and $i: Z \to X$ a closed immersion. Then for any quasi-coherent Cartier module $M$ on $X$ the cokernel of the natural inclusion $i_\ast R^0i^! M \subseteq H^0_Z(M)$ is locally nilpotent (see Definition \ref{Deflocnilpotent}). In particular, $i_\ast i^! \cong R\Gamma_Z$ in $D^b_{crys}(X)$.
\end{Le}
\begin{proof}
This is a local statement by \cite[Lemma 2.2.4]{blickleboecklecartiercrystals} so that we may reduce to $X = \Spec R$ being affine and $Z = \Spec R/I$ for an ideal $I$. Fix a section $m$ such that $I^k m = 0$ and note that also $N = R \cdot m$ is annihilated by $I^k$. Take any $e$ such that $p^e \geq k$. Then $I \kappa^e(N)  \subseteq \kappa^e(I^{p^e} N) = 0$ showing that the inclusion is a nil-isomorphism. The supplement follows by passing to derived functors since $QCrys$ is the localization of quasi-coherent Cartier modules at locally nilpotent ones.
\end{proof}

\begin{Ko}
Let $f: X \to Y$ be a smooth morphism of embeddable schemes. If $\Sigma_X$ denotes the equivalence from $D^b_{crys}(X) \to D^b_{unit\, F}(X)$ (see  Theorem \ref{CartierunitFequivalence}) and similarly for $\Sigma_Y$ then one has $f^! \circ \Sigma_Y \cong \Sigma_X \circ f^!$.
\end{Ko}
\begin{proof}
This proceeds in a similar fashion to Theorem \ref{ShriekSmoothEmbeddableCommute} using Lemma \ref{GammaDescription} instead of \cite[Lemma 3.2]{blicklestaeblerfunctorialtestmodules}.
\end{proof}

\section{Intermediate extensions and smooth pullbacks}
\label{IESection}

Throughout this section we fix some Cartier algebra $\mathcal{C}$ satisfying suitable finiteness conditions as discussed in Section \ref{Cartiercrystals}. Any Cartier module or crystal will be with respect to $\mathcal{C}$. When $\mathcal{C}$ is assumed to be principal we will indicate this by considering a Cartier module of the form $(M, \kappa)$. Recall the notion of \emph{intermediate extension} in Cartier crystals (see Definition \ref{IntermediateExtensionDef} and discussion thereafter).

\begin{Bem}
When dealing with intermediate extensions we use the assumption that $\mathcal{C}= \langle \kappa \rangle$ only to ensure that the upper shriek functor is well-defined for locally closed immersions (i.e.\ independent of the factorization). In the principal case we can exploit the equivalence with unit $F$-modules where the upper shriek corresponds to a pullback so that compatibilities are straightforward. This is probably also true for more general Cartier algebras but we do not verify this here.
\end{Bem}

\begin{Le}
\label{LocalizationIE}
Let $j: U \to \Spec R$ and $f: V \to \Spec R$ be open immersions. Consider the pull back square
\[ \begin{xy} \xymatrix{U \cap V \ar[r]^{j'} \ar[d]^{f'} & V \ar[d]^f\\ U \ar[r]^{j} & \Spec R} \end{xy}\]
Then $f^! j_{! \ast} \cong j'_{! \ast} f'^!$ in the category of Cartier crystals.
\end{Le}
\begin{proof}
Let $M$ be a Cartier crystal on $U$ and denote $j_{!\ast} M$ by $N$. The issue is local on $V$ so that we may assume that $V = \Spec R_x$ for some $x \in R$. Let $A \subseteq f^!N$ be such that $j'^! A = j'^! f^! N$. Denote the localization map $N \to N_x = f^! N$ by $\varphi$.

Choose an open affine covering $\bigcup_i D(y_i)$ of $U$. Since $A$ and $N_x$ restricted to $U \cap V$ coincide they coincide a fortiori when restricted to each $D(y_i) \cap V$. Hence, by \cite[Lemma 2.2]{blicklestaeblerfunctorialtestmodules} we have inclusions $y_i f^!N \subseteq A \subseteq f^! N$. Taking preimages along $\varphi$ we obtain $y_i N = y_i \varphi^{-1}(f^!N) \subseteq \varphi^{-1}(y_i f^!N) \subseteq \varphi^{-1} A \subseteq N$. Localizing at $y_i$ yields $\varphi^{-1}(A)_{y_i} = N_{y_i}$. Since the $D(y_i)$ form a covering of $U$ we conclude that $\varphi^{-1}(A)\vert_U = N\vert_U$. But since $N$ is minimal with this property $\varphi^{-1}(A) = N$. As $\varphi(\varphi^{-1}(A)) \subseteq A$ we conclude that $A = N_x$ since $A$ contains a system of generators.
\end{proof}

\begin{Prop}
\label{EtaleShriekIntermediateExtension}
Let $j: U \to Y$ be an open immersion, $f: X \to Y$ an \'etale morphism and let $M$ be a Cartier crystal on $U$ and let $N \subseteq j_\ast M$ such that $j^! N = M$. If $N$ is the intermediate extension of $M$, then $f^! N$ is the intermediate extension of $f'^! M$, where $f'$ is the pull back of $f$ along $j$. If $f$ is surjective, the converse holds.
\end{Prop}
\begin{proof}
Assume that $f^!N$ is not the intermediate extension. That is, there exists $A \subsetneq f^!N$ with $j'^! A = j'^! f^!N$. In particular, we find an open affine $\Spec R \subseteq Y$ and $\Spec S \subseteq f^{-1}(\Spec R)$ such that the inclusion $A \subsetneq f^!N$ restricted to $\Spec S$ is still proper. Using Lemma \ref{LocalizationIE} we have reduced to the situation where both $Y = \Spec R, X =\Spec S$ are affine. By factoring $X \to f(X) \to Y$ we may further assume that $f$ is surjective. It follows (\cite[Remark 2.19]{milne}) that $\alpha:N  \to f_\ast f^! N$ is injective. Fix a covering $D(y_i)$ of $U$ by open affines. By abuse of notation we will denote $f^{-1}(D(y_i))$ again by $D(y_i)$. Since $A$ and $f^!N$ restricted to each $D(y_i)$ agree we obtain from \cite[Lemma 2.2]{blicklestaeblerfunctorialtestmodules} an inclusion
\[ \mathcal{C} y_i f^! N = \mathcal{C} f^! y_i N \subseteq A \subseteq f^!N. \]
Note that by \cite[Lemma 6.1]{blicklestaeblerfunctorialtestmodules} $\mathcal{C} f^! y_iN = f^! \mathcal{C} y_i N$. Define the $\mathcal{C}_R$-module $A' = A \cap \alpha(N)$ and note that we have inclusions
\[\mathcal{C}_R y_i N \subseteq A' \subseteq N,\] where we identify $\alpha(N)$ with $N$. In particular, $A'$ and $N$ agree on $U$. Hence, we must have $A' = N$ since $N$ is the intermediate extension. It follows that $f^! A' = f^!N$. But $f^! A' \subseteq A \subsetneq f^! N$ which is a contradiction.

Assume now that $f$ is surjective. Then the converse follows from faithfully flat descent.
\end{proof}

\begin{Bem}
\label{IrredComponents}
Let us point out that in assuming that connected components are irreducible components we could also restrict our attention to irreducible schemes in proving that intermediate extensions commute with smooth pullbacks. Again, we do not impose this condition when proving that intermediate extensions commute with smooth pull backs in Cartier crystals but we do need it to apply the equivalence with perverse constructible sheaves.

Let $f: X \to Y$ be a smooth morphism and $i: U \to Y$ a locally closed immersion. Denote the connected (= irreducible) components of $Y$ by $Y_1, \ldots, Y_n$ and write $\alpha_e: Y_e \to Y$ for the inclusion.
Let us consider the following commutative diagram
\[ \begin{xy} \xymatrix{Y_e \ar[r]^{\alpha_e} &Y\\ Y_e \cap U \ar[u]^{i'} \ar[r]^{\alpha_e'}& U \ar[u]^i}\end{xy}\]
Then we claim that for any Cartier crystal $M$ on $U$ one has \begin{equation} \label{IrredCompReductionEq}i_{!*} M = \bigoplus_{e} \alpha'_{e \ast} i'_{!\ast} \alpha_e^! M.\end{equation}

By Lemma \ref{LclosedIEOpenIE} below we only need to look at the cases where $i$ is either a closed or open immersion. If $i$ is a closed immersion, then $i_{!\ast} = i_\ast$ and similarly for $i'$. Thus the right hand side of (\ref{IrredCompReductionEq}) is isomorphic to $\bigoplus_e i_\ast \alpha_{e \ast} \alpha_e^! M \cong i_\ast M$ as desired.

For $i$ an open immersion we use Proposition \ref{EtaleShriekIntermediateExtension} with $f = \alpha_e$ to conclude that $i_{!\ast} \alpha_e^! M = {\alpha'_e}^! i'_{!\ast}$. Since for a smooth morphism $f$ the functor $f^!$ commutes with direct sums one now readily obtains that $f^! i_{!\ast} = i'_{!\ast} f'^!$ from the irreducible case.
\end{Bem}

We now come to the general smooth case, which after some reductions can be handled in a similar manner as the \'etale case.

\begin{Theo}
\label{SmoothPullbackIE}
Let $f:X  \to Y$ be a smooth morphism of $F$-finite schemes and $j: U \to Y$ an open subset. Given a Cartier crystal $M$ with intermediate extension $j_{! \ast} M$ one has $j'_{! \ast} f'^! M \cong f^! j_{! \ast} M$ canonically, where $f'$ and $j'$ are the base changes of $f$ and $j$. Conversely, if $f$ is surjective and $f^! N$ is the intermediate extension of $f'^! M$, where $N$ is a Cartier crystal on $Y$, then $N$ is the intermediate extension of $M$.
\end{Theo}
\begin{proof}
First of all, note that $f^! j_{!\ast} M \subseteq f^! j_\ast M \cong j'_\ast f'^! M$ by \cite[Proposition 8.2]{blicklestaeblerfunctorialtestmodules}. Assume that it is not the intermediate extension. That is, there is $A \subsetneq f^! j_{!\ast} M$ with $j^! A = f'^! M$. In particular, we find an open affine $V$ such that $A\vert_V \subsetneq f^! j_{! \ast} M\vert_V$. Shrinking $V$ if necessary we find an open affine $W \subseteq X$ such that $f(W) \subseteq V$ and such that $f\vert_W$ factors as $W \to \mathbb{A}^n_V \to V$, where the first morphism is \'etale and the second  is the structural morphism. Using Proposition \ref{EtaleShriekIntermediateExtension} we only have to deal with the case $f: \mathbb{A}^n_V \to V$, where $V$ is affine. If $x_1, \ldots, x_n$ denote coordinates for $\mathbb{A}^n_V$ then the map \[\alpha: N \longrightarrow f_\ast f^! N, n \longmapsto n \otimes dx_1 \wedge \ldots \wedge dx_n\] is injective (see \cite[Remark 2.19]{milne}). Now one can reason just as in Proposition \ref{EtaleShriekIntermediateExtension} to derive a contradiction.
\end{proof}

\begin{Le}
\label{LclosedIEOpenIE}
Let $j: U \to X$ be a locally closed immersion which factors as $j = fg$ with $f$ a closed immersion and $g$ an open immersion. Then if $(M, \kappa)$ is a Cartier crystal on $U$, we have $j_{!\ast} M \cong f_\ast g_{!\ast} M$ canonically.
\end{Le}
\begin{proof}
Using \cite[Theorem 5.12 (c)]{schedlmeiercartierpervers} and \cite[Theorem 4.1.2]{blickleboecklecartiercrystals} we have \[j^! f_\ast g_{!\ast} M \cong g^! f^! f_\ast g_{!\ast}M \cong M.\] It remains to show minimality. If $A \subsetneq f_\ast g_{!\ast}M$ is a subcrystal, then $A$ is also supported in $Z$, where $f:Z \to X$. Indeed, if we denote the associated open immersion of the complement by $v: V \to X$, then $v^! f_\ast g_{!\ast}M = 0$ and $v^!$ is exact so that $v^! A = 0$. Hence, we have $A = f_\ast B$ for some crystal $B$ on $Z$. Then \[B = f^! f_\ast B \subsetneq f^! f_\ast g_{!\ast} M = g_{!\ast} M,\] where we use that $f^! f_\ast = id$. As $g_{!\ast} M$ is minimal we conclude that $g^! B \subsetneq M$. Hence, $f_\ast g_{!\ast}M$ is minimal.
\end{proof}

Note that this shows in particular existence of intermediate extensions for locally closed immersions. Also note that this implies, with \cite[Proposition 6.18]{schedlmeiercartierpervers}, that any simple Cartier crystal on $X$ is of the form $j_{!\ast}M$ for $j: U \to X$ a locally closed immersion.

We come to the main result

\begin{Theo}
Let $f:X \to Y$ be a smooth morphism of $F$-finite embeddable schemes and $j: U \to Y$ a locally closed immersion. If we denote the base change of $j$ by $j'$ and the base change of $f$ by $f'$ then the functors $j'_{! \ast} f'^{-1}$ and $f^{-1} j_{!\ast}$ from perverse constructible $\mathbb{F}_p$-sheaves on $U$ to perverse constructible $\mathbb{F}_p$-sheaves on $Y$ are naturally isomorphic.
\end{Theo}
\begin{proof}
By \cite[Corollary 6.21]{schedlmeiercartierpervers} and Corollary \ref{CrystoPervSmoothEmbeddable} it is sufficient to prove the corresponding statement in Cartier crystals, where the Cartier algebra is of the form $\mathcal{C} = R\langle\kappa\rangle$, i.e.\ principally generated. We may factor $j$ as $gh$, where $g$ is a closed immersion and $h$ is an open immersion. Moreover, by Lemma \ref{LclosedIEOpenIE} we may treat these two cases separately. For a closed immersion this is \cite[Lemma 4.5]{staeblerunitftestmoduln} and for an open immersions it follows from Theorem \ref{SmoothPullbackIE}. 
\end{proof}

\bibliography{bibliothek.bib}
\bibliographystyle{amsalpha}
\end{document}